\documentclass [final]{amsart}

\newcommand{\eqdef  }{\overset{\mbox{\tiny{def}}}{=}}

\usepackage{amsmath,amsfonts,amssymb,graphicx,latexsym,cite,hyperref,xcolor}

\newcommand{\rth}{{\mathbb{R}^3}}
\usepackage{todonotes}

\newtheorem{theorem}{Theorem}

\newtheorem{lemma}[theorem]{Lemma}
\newtheorem{proposition}{Proposition}

\theoremstyle{definition}

\begin{document}
\title[$L^2$ decay for the linearized Landau equation]{$L^2$ decay for the linearized Landau equation with the specular boundary condition}
\author{Yan Guo}
\address{Brown University, Providence RI 02912, USA }
\email{yan\_guo@brown.edu}

\author{Hyung Ju Hwang}
\address{Department of Mathematics, POSTECH, Pohang 37673, Republic of Korea } 
\email{hjhwang@postech.ac.kr}

\author{Jin Woo Jang}
\address{Institute for Applied Mathematics, University of Bonn, 53115 Bonn, Germany}
\email{jangjinw@iam.uni-bonn.de}

\author{Zhimeng Ouyang}
\address{Brown University, Providence RI 02912, USA }
\email{zhimeng\_ouyang@brown.edu}

	\begin{abstract}
		In this paper, we develop an alternative approach to establish the $L^2$ decay estimate for the linearized Landau equation in a bounded domain with specular boundary condition. The proof is based on the methodology of proof by contradiction motivated by \cite{MR1908664} and \cite{MR2679358}.
	\end{abstract}
\maketitle
\section{Introduction}We consider the following linearized \textit{Landau equation}
 \begin{equation}\label{linear}
 \partial_tf+v\cdot \nabla_x f+Lf=\Gamma(g,f).
 \end{equation} The linear operator $L$ is defined as%
\begin{equation}		\label{L}
L = -A -K,
\end{equation} where the linear operator $A$  consists of the terms with at least one momentum derivative on $f$ as$$
Af \eqdef \mu^{-1/2}\partial_{i}\left\{  \mu^{1/2}\sigma^{ij}[\partial_{j}f+v_{j}f]\right\} =\partial_{i}[\sigma^{ij}\partial_{j}f]-\sigma^{ij}v_{i}v_{j}f+\partial_{i}\sigma^{i}f,
$$ and the linear operator $K$ consists of the rest of the operator $L$ which does not contain any momentum derivative of $f$ as$$
Kf\eqdef -\mu^{-1/2}\partial_{i}\left\{  \mu\left[  \phi^{ij}\ast\left\{  \mu^{1/2}[\partial_{j}f+v_{j}f]\right\}  \right]  \right\}.
$$ Note that the momentum derivative $\partial_jf$ inside $Kf$ can always be moved to $\mu^{1/2}$ and outside the convolution by the chain rule and a property of a convolution operator. On the other hand, the nonlinear operator $\Gamma$ is defined as
\begin{equation}		\label{Gamma}
\begin{split}
	\Gamma\lbrack g,f]  &  \eqdef \partial_{i}\left[  \left\{  \phi^{ij}\ast\lbrack\mu^{1/2}g]\right\}  \partial_{j}f\right]  -\left\{  \phi^{ij}\ast\lbrack	v_{i}\mu^{1/2}g]\right\}  \partial_{j}f\\
	&  \quad-\partial_{i}\left[  \left\{  \phi^{ij}\ast\lbrack\mu^{1/2}%
	\partial_{j}g]\right\}  f\right]  +\left\{  \phi^{ij}\ast\lbrack v_{i}\mu^{1/2}\partial_{j}g]\right\}  f,
\end{split}
\end{equation}where the diffusion matrix (collision frequency) $\sigma^{ij}_u$ is defined as
$$
\sigma_{u}^{ij}(v) \eqdef  \phi^{ij}*u = \int_{\mathbb{R}^{3}}\phi^{ij}(v-v^{\prime})u(v^{\prime})dv^{\prime}.
$$We also denote the special case when $u=\mu$ as
$$
\sigma^{ij} = \sigma^{ij}_{\mu}, \quad \sigma^{i} = \sigma^{ij}v_{j}.
$$
\subsection{The initial and boundary conditions}
 The initial-boundary conditions of $f$ for the specular reflection boundary that we consider are given by
 \begin{equation}\label{ib}
 \begin{cases}
 f(0,x,v)=f_0(x,v), \text{ if }x\in\Omega \text{ and }v\in\rth,\\
 f(t,x,v)=f(t,x,v-2(v\cdot n_x)n_x),\text{ if }x\in\partial\Omega\text{ and }v\cdot n_x<0,
 \end{cases}
 \end{equation} for some $\|f_0\|_{\infty,\vartheta+m}\le \epsilon,$ for some small $\epsilon>0$, $\vartheta\ge 0$ and $m>\frac{3}{2}.$ Throughout this paper, our domain $\Omega=\{x:\zeta(x)<0\}$ is connected and bounded with $\zeta(x)$ being a smooth function. We also assume that $\nabla\zeta(x) \neq 0$ at the boundary $\zeta(x)=0.$ We define the outward normal vector $n_x$ on the boundary $\partial \Omega$ as 
$$
n_x\eqdef  \frac{\nabla \zeta(x)}{|\nabla\zeta(x)|}.
$$We say that the domain $\Omega$ is rotationally symmetric if there exist vectors $x_0$ and $w$ such that
    $$((x-x_0)\times w)\cdot n_x =0,$$ for all $x\in\partial\Omega.$
    Without loss of generality, we assume that the conservation laws of total mass and energy for $t\ge 0$ terms of the perturbation $f$:
    \begin{equation}
    \label{conservationlaws}
        \begin{split}
           \int_{\Omega\times \rth}f(t,x,v)\sqrt{\mu}dxdv =0,\ \ 
           \int_{\Omega\times \rth}|v|^2f(t,x,v)\sqrt{\mu}dxdv &=0.
        \end{split}
    \end{equation}In addition, we assume the conservation of total angular momentum if $\Omega$ is rotationally symmetric: \begin{equation}
    \label{conservationangcon}
        \begin{split}
           \int_{\Omega\times \rth}((x-x_0)\times w)\cdot f(t,x,v)v\sqrt{\mu}dxdv &=0.
        \end{split}
    \end{equation}Define the energy\begin{equation}		\label{E}
    \mathcal{E}_{\vartheta}(f(t)) \eqdef   \big\|  f(t)\big\|_{2,\vartheta}^{2} + \int_{0}^{t} \big\|  f(s)\big\|_{\sigma,\vartheta}^{2} ds,
\end{equation} 
where $$\Vert f\Vert_{p,\vartheta}^{p}\eqdef  \int_{\Omega\times\mathbb{R}^{3}}(1+|v|)^{p\vartheta}|f|^{p}dxdv$$ and	$$\Vert f\Vert_{\sigma,\vartheta}^{2}\eqdef  \iint_{\Omega\times\mathbb{R}^{3}}(1+|v|)^{2\vartheta}\left[  \sigma^{ij}\partial_{i}f\partial_{j}f+\sigma^{ij}v_{i}v_{j}f^{2}\right]  dxdv.$$
\subsection{Main theorem}
We now introduce our main theorem on the $L^2$ decay estimates for the weak solutions $f$ to \eqref{linear}.
\begin{theorem}[Theorem 13 of \cite{ARMAjang}]\label{linear l2}
Let $f$ be the weak solution of \eqref{linear} with initial-boundary value conditions \eqref{ib}, which satisfies the conservation laws \eqref{conservationlaws}, and \eqref{conservationangcon} if $\Omega$ has a rotational symmetry. 
Suppose that $\|g\|_{L^\infty_m} < \epsilon$ for some $\epsilon>0$ and $m> 3/2$.
For any $\vartheta\in 2^{-1}\mathbb{N} \cup\{0\}$, there exist $C$ and $\epsilon=\epsilon(\vartheta)>0$ such that 
\begin{equation}	\label{Eq : energy estimate linear}
	\sup_{0 \le s< \infty}\mathcal{E}_{\vartheta}(f(s)) \le C 2^{2\vartheta} \mathcal{E}_{\vartheta}(f_0),
\end{equation}
and
\begin{equation}	\label{Eq : decay estimate linear}
	\|f(t)\|_{2,\vartheta} \le C_{\vartheta,k} \left(\mathcal{E}_{\vartheta+\frac{k}{2}}(0) \right)^{1/2}\left(  1+ \frac{t}{k}\right)^{-k/2}%
\end{equation}
for any $t>0$ and $k\in\mathbb{N}$, where $\mathcal{E}_\vartheta(f(t))$ is defined as \eqref{E}.
\end{theorem}
  In order to prove Theorem \ref{linear l2}, it is crucial to obtain the following positivity of $L$:
    \begin{proposition}\label{main coercivity}Let $f$ be a weak solution of \eqref{linear}-\eqref{conservationangcon} with  $\mathcal{E}_\vartheta(f(0))$ bounded for some $\vartheta\ge 0.$ Then there exists a sufficiently small positive constant $\epsilon>0$ such that if \begin{equation}\label{gsmall ep}\|g\|_{L^\infty_m}\le \epsilon,\end{equation}for some $m>\frac{3}{2}$, then we have $\delta_\epsilon>0$ such that 
    $$\int_0^1 (Lf,f)ds\ge \delta_\epsilon \int_0^1 \|f\|^2_{\sigma}ds.$$
    \end{proposition}The proof for this Proposition will be given in the next section.
\section{Positivity of $L$}
   In order to prove the positivity of $L$, it suffices to prove the following proposition as we have Lemma 5 of \cite{guo2002landau}:
    \begin{proposition}	\label{main proposition}Let $f$ be a weak solution of \eqref{linear}-\eqref{conservationangcon} with $\mathcal{E}_\vartheta(f(0))$ bounded for some $\vartheta\ge 0.$ Then there exists a sufficiently small $\epsilon>0$ such that if $\|g\|_{L^\infty_m} \le \epsilon$ for some $m>\frac{3}{2}$,
 we have $C_\epsilon>0$ such
 	that 
 	\[
 	\int_{0}^{1} \|P f(\tau)\|_{\sigma}^{2}ds \le C_\epsilon \int
 	_{0}^{1} \|(I-P)f(\tau)\|_{\sigma}^{2}ds.
 	\]
    \end{proposition}
    \begin{proof}
    If the proposition is not true, then there exist a sequence of family $g_n$ and a sequence of solutions $f_n$ to \eqref{linear}-\eqref{conservationangcon} with $g=g_n$ and $f=f_n$ such that \begin{equation}\label{gsmall}\|g_n\|_{L^\infty_m}\le \frac{1}{n},\end{equation} for some $m>\frac{3}{2}$, but 
 	\begin{equation}\label{contradiction} \int
 	_{0}^{1} \|(I-P)f_n(\tau)\|_{\sigma}^{2}ds\le \frac{1}{n}\int_{0}^{1} \|P f_n(\tau)\|_{\sigma}^{2}ds ,\end{equation} for any $n$.

We first prove the weak compactness of $f_n$.  
We first reformulate the equation \eqref{linear} as
\begin{equation}
	f_{t}+v\cdot\nabla_{x}f=\bar{A}_{g}f+\bar{K}_{g}f, \label{rearrange Landau}%
\end{equation}%
\begin{equation}		\label{Eq : bar A}
	\begin{split}
		\bar{A}_{g}f  &  :=\partial_{i}\left[  \left\{  \phi^{ij}\ast\lbrack\mu	+\mu^{1/2}g]\right\}  \partial_{j}f\right]\\
		&  \quad-\left\{  \phi^{ij}\ast\lbrack v_{i}\mu^{1/2}g]\right\}  \partial_{j}f-\left\{  \phi^{ij}\ast\lbrack\mu^{1/2}\partial_{j}g]\right\}
		\partial_{i}f\\
		&  =: \nabla_{v}\cdot(\sigma_{G}\nabla_{v}f)+a_{g}\cdot\nabla_{v}f,
	\end{split}
\end{equation}
\begin{equation}		\label{Eq : bar K}
	\begin{split}
		\bar{K}_{g}f  &  :=Kf+\partial_{i}\sigma^{i}f  -\sigma^{ij}v_{i}v_{j}f\\
		&  \quad-\partial_{i}\left\{  \phi^{ij}\ast\lbrack\mu^{1/2}\partial	_{j}g]\right\}  f+\left\{  \phi^{ij}\ast\lbrack v_{i}\mu^{1/2}\partial_{j}g]\right\}  f,
	\end{split}
\end{equation}
where $G=\mu+\sqrt{\mu}g$, \begin{equation}		\label{Eq : K}
	Kf:=-\mu^{-1/2}\partial_{i}\left\{  \mu\left[  \phi^{ij}\ast\left\{  \mu^{1/2}[\partial_{j}f+v_{j}f]\right\}  \right]  \right\}  ,
\end{equation}  and $
 \sigma^{ij} = \sigma^{ij}_{\mu}, \quad \sigma^{i} = \sigma^{ij}v_{j},
$ with $$
 \sigma_{u}^{ij}(v) \eqdef  \phi^{ij}*u = \int_{\mathbb{R}^{3}}\phi^{ij}(v-v^{\prime})u(v^{\prime})dv^{\prime}.
 $$
 Note that the eigenvalues $\lambda(v)$ of $\sigma(v)$ satisfy \cite[Lemma 3]{guo2002landau} \begin{equation}\label{spectrumbound}(1+|v|)^{-3}\lesssim \lambda(v) \lesssim (1+|v|)^{-1}.\end{equation}
	For any fixed $l<0$, we multiply \eqref{linear} for $f=f_n$ and $g=g_n$ by $(1+|v|)^{2l}f_n$ and integrate both sides of the resulting equation and obtain
	\begin{multline}		\notag
		\iint_{\Omega \times \mathbb{R}^3} \frac{1}{2} \left( (1+|v|)^{2l}f_n^2(t,x,v) - (1+|v|)^{2l}f_n^2(t_0,x,v) \right) dxdv \\+ \int_{t_0}^t \iint_{\Omega\times \mathbb{R}^3}  (1+|v|)^{2l}(Lf_n)f_n dxdvds=\int_{t_0}^t \iint_{\Omega\times \mathbb{R}^3}(1+|v|)^{2l}\Gamma(g,f_n)f_n dxdvds\\
		\le \int_{t_0}^t\|g_n\|_\infty \|f_n\|^2_{\sigma,l}ds,
	\end{multline} by Theorem 2.8 of \cite{2016arXiv161005346K}. Also, since $l<0,$ we deduce by Lemma 6 of \cite{guo2002landau} that $$\int_{t_0}^t \iint_{\Omega\times \mathbb{R}^3}  (1+|v|)^{2l}(Lf_n)f_n dxdvds\ge \int_{t_0}^tds\ \left(\frac{1}{2}\|f_n(s)\|_{\sigma,l}^2 -C_l \|(1+|v|)^lf_n(s)\|^2_{L^2}\right).$$
Thus, we have
	\begin{multline}		\notag
 \frac{1}{2}  \|(1+|v|)^lf_n(t)\|^2_{L^2} +\int_{t_0}^tds\ \frac{1}{4}\|f_n(s)\|_{\sigma,l}^2\\ \le 	 \|(1+|v|)^lf_n(t_0)\|^2_{L^2}+ C \int_{t_0}^tds \|(1+|v|)^lf_n(s)\|^2_{L^2}.
\end{multline}
	Thus, by \eqref{gsmall} and the Gr\"onwall inequality, we obtain that
		 \begin{equation}\label{fnbounded}\|(1+|v|)^lf_n(t)\|_{L^2}^2+\int_{t_0}^t \|f_n(s)\|^2_{\sigma,l}\  ds\le C e^{t-t_0} \|(1+|v|)^lf_n(t_0)\|^2_{L^2}. \end{equation}
		 On the other hand, we note that
		 $$\|f\|_\sigma \ge C \|(1+|v|)^{-1/2}f\|_{L^2},$$ by \eqref{spectrumbound}. Thus we have
		 \begin{multline*}
		 	\frac{d}{dt}\int_{t_0}^t\|f_n(s)\|^2_\sigma ds=\|f_n(t)\|^2_\sigma \ge C\|(1+|v|)^{-1/2}f_n(t)\|^2_{L^2}\\\ge C\|(1+|v|)^{-1/2}f_n(t_0)\|^2_{L^2}\\-2C\int_{t_0}^t \iint_{\Omega\times \mathbb{R}^3}  (1+|v|)^{-1}(Lf_n)f_n dxdvds-2C\int_{t_0}^t\|g_n\|_\infty \|f_n\|^2_{\sigma,-1/2}ds\\
		 	\ge C\|(1+|v|)^{-1/2}f_n(t_0)\|^2_{L^2}-2C\int_{t_0}^t\left(\frac{3}{2}\|f(s)\|_{\sigma,-1/2}^2 -C \|f(s)\|^2_\sigma\right)ds\\-2C\int_{t_0}^t\|g_n\|_\infty \|f_n\|^2_{\sigma,-1/2}ds
		 	\ge C\|(1+|v|)^{-1/2}f_n(t_0)\|^2_{L^2}-C'\int_{t_0}^t \|f(s)\|^2_\sigma ds ,
		 \end{multline*}for some $C'>0$ by Lemma 2.7 of \cite{2016arXiv161005346K}.
	 By \eqref{gsmall} and the Gr\"onwall inequality, we obtain that \begin{equation}\label{85}\int_{t_0}^t \|f_n(s)\|^2_\sigma ds \ge 
	 C(1-e^{-C'(t-t_0)})\|(1+|v|)^{-1/2}f_n(t_0)\|^2_{L^2}.\end{equation}
Now we define the normalized term $Z_n$ of $f_n$ as
$$Z_n\eqdef \frac{f_n}{\sqrt{\int_0^1 \|Pf_n\|_\sigma^2 ds }}.$$
For $s\in [0,1]$, we have 
$$
\|(1+|v|)^{-1/2}Z_n(s)\|^2_{L^2}
=\frac{\|(1+|v|)^{-1/2}f_n(s)\|^2_{L^2}}{\sqrt{\int_0^1\|Pf_n\|^2_\sigma d\tau}}\le \frac{C e^{s}\|(1+|v|)^{-1/2}f_n(0)\|^2_{L^2}}{\int_0^1\|Pf_n\|^2_\sigma d\tau},
$$by \eqref{fnbounded} for $l=-1/2$. 
On the other hand, by the assumption \eqref{contradiction} we have
\begin{multline*}
	(n+1)  \int_0^1\|Pf_n\|^2_\sigma d\tau\ge n \int_0^1\|Pf_n\|^2_\sigma d\tau+n\int_0^1\|(1-P)f_n\|^2_\sigma d\tau\ge n \int_0^1\|f_n\|^2_\sigma d\tau.
\end{multline*}
Thus,
\begin{multline*}
\|(1+|v|)^{-1/2}Z_n(s)\|^2_{L^2}
=\frac{\|(1+|v|)^{-1/2}f_n(s)\|^2_{L^2}}{\sqrt{\int_0^1\|Pf_n\|^2_\sigma d\tau}}\\\le \frac{n+1}{n}\frac{C e^{s}\|(1+|v|)^{-1/2}f_n(0)\|^2_{L^2}}{\int_0^1\|f_n\|^2_\sigma d\tau} \le 2\frac{C \|(1+|v|)^{-1/2}f_n(0)\|^2_{L^2}}{\int_0^1\|f_n\|^2_\sigma d\tau},
\end{multline*} for any $n\ge 1.$
Now, by \eqref{85}, we have
$$\int_0^s \|f_n(\tau)\|^2_\sigma d\tau \ge 
C(1-e^{-Cs})\|(1+|v|)^{-1/2}f_n(0)\|^2_{L^2}.$$
Thus, we obtain the uniform bound $$\sup_{0\le s\le 1}\|(1+|v|)^{-1/2}Z_n(s)\|^2_{L^2}\le C$$ for some $C>0$. Also, by the normalization we already had $\int_0^1 \|Z_n(s)\|^2_\sigma ds=1$.
	Note that this will also imply that there is no concentration in time. Therefore, there exists the weak limit $Z$ of $Z_n$ in $\int_0^1 \|\cdot \|_\sigma^2 ds$. Also, by \eqref{contradiction}, we have 
\begin{equation}\label{I-P limit}\int_0^1 \|(I-P)Z_n \|_\sigma^2 ds\le \frac{1}{n}\rightarrow 0.\end{equation}
    By the triangle inequality, we also have that $\int_0^1 \|PZ_n(s)\|^2_\sigma ds$ is uniformly bounded from above. In addition, the norm $\|\cdot\|_\sigma$ is an anisotropic Sobolev norm with respect to direction of the velocity $v$ by definition. Since the eigenvalues $\lambda(v)$ of the matrix $\sigma(v)$ satisfies the bound \eqref{spectrumbound}, the normed vector space with the norm $\|\cdot \|_\sigma$ can be understood as a weighted $L^2$ Sobolev space and is reflexive. Then by Alaoglu's theorem and Eberlein–Šmulian's theorem, $PZ_n$ converges weakly to $PZ$ in $\int_0^1 \|\cdot \|_\sigma^2 ds$ up to a subsequence.
    Thus, we conclude that $(I-P)Z =0$ and $Z=PZ$.
    Thus, we can write $Z(t,x,v)$ as
    $$Z(t,x,v)=(a(t,x)+b(t,x)\cdot v + c(t,x)|v|^2)\sqrt{\mu}.$$
    Also, by taking the limit $n\rightarrow \infty$, we note that the limit $Z$ satisfies 
    \begin{equation}\label{limit eq}\partial_t Z+v\cdot \nabla_x Z=\Gamma(g_\infty,Z)=0\end{equation} in the sense of distribution as the condition \eqref{gsmall} makes $g_\infty =0$ a.e. outside a null set that results in the vanishing integral $\int \Gamma(g_\infty,Z)\phi$ via an integration by parts and we also have $\int LZ\phi$ vanishes as $Z=PZ$, for a test function $\phi\in C^1_c$.
    
    Now our main strategy is to show that $Z$ has to be zero by \eqref{I-P limit}, the specular reflection boundary conditions, \eqref{limit eq}, and the conservation laws \eqref{conservationlaws} and \eqref{conservationangcon}. On the other hand, we will show the strong convergence of $Z_n$ to $Z$ in $\int_0^1 \|\cdot \|_\sigma^2 ds$ by proving the compactness. This will lead us to a contradiction.
    
    We first introduce the following lemma which provides more information on the form of $Z$:
    \begin{lemma}[Lemma 6 of \cite{MR2679358}]\label{guolemma6}
    There exist constants $a_0,c_1,c_2$, and constant vectors $b_0,b_1$ and $\bar{w}$such that $Z(t,x,v)$ takes the form:
    \begin{multline*}
        \bigg(\left(\frac{c_0}{2}|x|^2-b_0\cdot x +a_0\right)+(-c_0tx-c_1x+\bar{w}\times x +b_0t+b_1)\times v\\ +\left(\frac{c_0t^2}{2}+c_1t+c_2\right)|v|^2\bigg)\sqrt{\mu}. 
    \end{multline*}Moreover, these constants are finite.
    \end{lemma}Our case also shares the same transport equation \eqref{limit eq} for $Z$ that deduces the same macroscopic equations as (72)-(76) of \cite{MR2679358} with $Z=PZ$ and the lemma holds. Moreover, a better bound \eqref{fnbounded} provides that the coefficients are finite.
    
    \subsection{Plan for the proof of the strong convergence}
    We first show the strong convergence of $Z_n$ to $Z$ in $\int_0^1 \|\cdot \|_\sigma^2 ds$. 
    First of all, we note that we have seen already that there is no concentration in time-boundary at $s=0$ or $s=1$ by \eqref{fnbounded}.
Then regarding the remainder of the domain $(\varepsilon,1-\varepsilon)\times \Omega\times \rth$ for some $\varepsilon>0$, we split it into three parts;  we define the interior $D^\varepsilon_{int}$, the non-grazing set $D^\varepsilon_{ng}$, and the singular grazing set $D^\varepsilon_{sg}$ so that 
$$(\varepsilon,1-\varepsilon)\times \Omega\times \rth=D^\varepsilon_{int}\cup D^\varepsilon_{lv}\cup D^\varepsilon_{ng}\cup D^\varepsilon_{sg}.$$ More precisely, we define the interior $D^\varepsilon_{int}$ as
    $$D^\varepsilon_{int} \eqdef (\varepsilon,1-\varepsilon)\times S_\varepsilon ,$$ where 
    $$S_\varepsilon = \left\{(x,v)\in \Omega\times \rth: \zeta(x)<-\varepsilon^4\text{ and }|v|\le \frac{4}{\varepsilon}\right\} .$$ 
    Then we define sets of the compliment.
    Firstly, define the set of large velocity $D^\varepsilon_{lv}$ as $$D^\varepsilon_{lv} \eqdef (\varepsilon,1-\varepsilon)\times \Omega\times \left\{|v|>\frac{4}{\varepsilon}\right\} \eqdef (\varepsilon,1-\varepsilon)\times S^c_{\varepsilon,0}.$$ 
    We define the singular grazing set 
    $D^\varepsilon_{sg}$ as $$D^\varepsilon_{sg} \eqdef (\varepsilon,1-\varepsilon)\times S^c_{\varepsilon,1} ,$$ where 
    $$S^c_{\varepsilon,1} = \left\{(x,v)\in \Omega\times \rth: \zeta(x)\ge -\varepsilon^4  \text{ and } \left[|n_x\cdot v| <\frac{\varepsilon}{2}\text { or } |v|> \frac{1}{\varepsilon}\right] \right\} .$$ 
Lastly, we define the non-grazing set 
    $D^\varepsilon_{ng}$ as $$D^\varepsilon_{ng} \eqdef (\varepsilon,1-\varepsilon)\times S^c_{\varepsilon,2} ,$$ where 
    $$S^c_{\varepsilon,2} = \left\{(x,v)\in \Omega\times \rth: \zeta(x)\ge -\varepsilon^4  \text{ and } \left[|n_x\cdot v| \ge\frac{\varepsilon}{2}\text { and } |v|\le \frac{1}{\varepsilon}\right] \right\} .$$ 
    Here recall that $\zeta(x) $ is the smooth function such that $\Omega =\{x:\zeta(x)<0\}.$

To prove the strong convergence in $\int_0^1 \|\cdot \|^2_\sigma ds,$ it suffices to show 
$$\sum_{1\le j\le 5} \int_0^1 ds \|\langle Z_n,e_j\rangle e_j - \langle Z,e_j\rangle e_j\|_{\sigma}^2  \rightarrow 0,$$ where $e_j$ are an orthonormal basis for 
$$\text{span}\{\sqrt{\mu},v\sqrt{\mu},|v|^2\sqrt{\mu}\},$$ as we have \eqref{I-P limit}.
Since $e_j(v)$ is smooth and the 0$^{th}$ and the 1$^{st}$ derivatives are exponentially decaying for large $|v|$, it suffices to prove 
$$\int_0^1 ds \int_\Omega dx\ |\langle Z_n,e_j\rangle - \langle Z,e_j\rangle |^2 \rightarrow 0.$$ We establish this by considering the decomposition of the domain as above.

\subsection{Interior compactness on $D^\varepsilon_{int}$}
Suppose $\chi_1$ is a smooth cutoff function that is supported on $D^\varepsilon_{int}$ and consider $$Z_n=(1-\chi_1)Z_n + \chi_1 Z_n.$$ In this subsection, we will consider the contribution $\chi_1 Z_n$ via the averaging lemma. We define another smooth cutoff function $\tilde{\chi}_1$ such that $\tilde{\chi}_1=1$ on $D^\varepsilon_{int}$ and $\tilde{\chi}_1=0$ outside $D^{\varepsilon/2}_{int}$. Then $\tilde{\chi}_1$ has a larger support than $\chi_1$ and $\tilde{\chi}_1 =1$ on $D^\varepsilon_{int}.$ The reason that we additionally define $\tilde{\chi}_1$ with a larger support than $\chi_1$ is in order to make $(1-\chi_1)Z_n=Z_n$ outside $D^\varepsilon_{int}$ and to make $\tilde{\chi}_1 Z_n =Z_n$ on $D^\varepsilon_{int}$.

We first observe that $\tilde{\chi}_1 Z_n$ satisfies the following equation
\begin{multline*}(\partial_t +v\cdot \nabla_x )(\tilde{\chi}_1 (1+|v|)^{-1/2}Z_n)\\=(1+|v|)^{-1/2}\left(-\tilde{\chi}_1 L[Z_n]+Z_n[\partial_t +v\cdot \nabla_x ]\tilde{\chi}_1 
+ \tilde{\chi}_1 \Gamma(g_n,Z_n)\right).\end{multline*}
We claim that the right-hand side is uniformly bounded in $L^2([0,1]\times \Omega\times \rth)$. 
We observe that the second term is easily uniformly bounded by the $L^2$ norm of $(1+|v|)^{-1/2}Z_n$, which is uniformly bounded by \eqref{fnbounded}. 
We also observe that the $L^2$ norms of the first and the third terms are bounded as follows.
By Lemma 1 of \cite{guo2002landau}, $\tilde{\chi}_1LZ_n$ can be written as
\begin{multline}\label{chiLZn}(1+|v|)^{-1/2}\tilde{\chi}_1LZ_n
\\=\bigg(-\partial_i(\sigma^{ij}\partial_jZ_n\tilde{\chi}_1)+\sigma^{ij}\partial_jZ_n\partial_i\tilde{\chi}_1-\partial_i\sigma^iZ_n\tilde{\chi}_1+\sigma^{ij}v_iv_jZ_n\tilde{\chi}_1\\
+\partial_i(\mu^{1/2}(\phi^{ij}\ast (\mu^{1/2}(\partial_j Z_n+v_jZ_n)))\tilde{\chi}_1)\\
-\mu^{1/2}(\phi^{ij}\ast (\mu^{1/2}(\partial_jZ_n+v_jZ_n)))\partial_i\tilde{\chi}_1\\
-v_i\mu^{1/2}(\phi^{ij}\ast (\mu^{1/2}(\partial_j Z_n+v_jZ_n)))\tilde{\chi}_1\bigg)(1+|v|)^{-1/2}
\equiv \partial_i g_1+g_2,\end{multline}where $\tilde{\chi}_1$ has a compact support and $g_1,g_2\in L^2([0,1]\times \Omega\times \rth)$ as $$\|  g_1\|_{L^2}+\|g_2\|_{L^2}\lesssim \|(I-P)Z_n\|_{\sigma}.$$ Also, we apply Lemma 7 at (56) of \cite{guo2002landau} to estimate $\tilde{\chi}_1\Gamma(g_n,Z_n)$ with $g_1$ there is our $g_n$ and $g_2=Z_n$ to see that
$$(1+|v|)^{-1/2}\tilde{\chi}_1\Gamma(g_n,Z_n)=\partial_{ij}g^{ij}+\partial_ig^i+g,$$ where
$$\|g^{ij}\|_{L^2}+\|g^i\|_{L^2}+\|g\|_{L^2}\lesssim \|g_n\|_{L^2}\|Z_n\|_\sigma\lesssim \|g_n\|_{L^\infty_m}\|Z_n\|_\sigma,$$as $m>\frac{3}{2}$ by the assumption \eqref{gsmall}.
Therefore, we have$$(\partial_t +v\cdot \nabla_x )(\tilde{\chi}_1(1+|v|)^{-1/2} Z_n)=h,$$ where $h\in L^2([0,1]\times \Omega; H^{-2}(\rth)).$ 
Then by the averaging lemma \cite[Theorem 5]{MR1003433}, we have
$$\langle \tilde{\chi}_1(1+|v|)^{-1/2} Z_n,e_j\rangle \in H^{1/6}([0,1]\times \Omega),$$ which holds uniformly in $n.$ Thus, up to a subsequence, we have the convergence
\begin{equation}\label{final1}\langle \tilde{\chi}_1(1+|v|)^{-1/2} Z_n,e_j\rangle \rightarrow \langle \tilde{\chi}_1 (1+|v|)^{-1/2}Z,e_j\rangle \text { in }L^2([0,1]\times \Omega).\end{equation}

    \subsection{Near the time-boundary and the grazing set $D^\varepsilon_{sg}$}Now, note that the leftover from the previous section is now
    $$\int_0^1 ds \int_\Omega dx |\langle (1-\chi_1)(Z_n-Z),e_j\rangle|^2  .$$
Regarding the contribution, we note that
\begin{multline}\label{final2}
   \int_0^1 ds \int_\Omega dx |\langle (1-\chi_1)|Z_n-Z|,e_j\rangle|^2\le  \int_0^1 ds \int_\Omega dx \int_\rth dv (1-\chi_1)^2|Z_n-Z|^2e_j \\
 = \left(\int_0^{\varepsilon}+\int_{1-\varepsilon}^1 \right)ds \int_{\Omega\times \rth}dxdv+\sum_{j=0}^2\int_0^1 ds    \int_{S^c_{\varepsilon,j}}dxdv.
\end{multline}
In this subsection, we only consider the contribution \begin{equation}\label{near the boundary ineq0}\left(\int_0^{\varepsilon}+\int_{1-\varepsilon}^1 \right)ds \int_{\Omega\times \rth}dxdv+\sum_{j=0}^1\int_0^1 ds    \int_{S^c_{\varepsilon,j}}dxdv,\end{equation}
 near the time-boundary and the grazing set $D^\varepsilon_{sg}$.

The first integral of \eqref{near the boundary ineq0} is bounded as\begin{multline*}
\left(\int_0^{\varepsilon}+\int_{1-\varepsilon}^1 \right)ds \int_{\Omega\times \rth}dxdv\ (1-\chi_1)^2|Z_n-Z|^2e_j
\\ \le 2\varepsilon\sup_{0\le s\le 1}(\|(1+|v|)^{-1/2}Z_n(s)\|^2_2+\|(1+|v|)^{-1/2}Z(s)\|^2_2).
\end{multline*} Note that we have the uniform boundedness
$$\sup_{0\le s\le 1,\ n\ge 1}\|(1+|v|)^{-1/2}Z_n(s)\|^2_{L^2}<C,$$ by \eqref{fnbounded} and that 
$\|(1+|v|)^{-1/2}Z(s)\|^2_2=\|(1+|v|)^{-1/2}Z(0)\|^2_2,$ by the transport equation \eqref{limit eq}. Then this rules out the possible concentration at $t=0$ or $t=1$. 

Regarding another term in \eqref{near the boundary ineq0},
we observe that
\begin{multline*}\int_0^1 ds    \int_{S^c_{\varepsilon,0}}dxdv\ (1-\chi_1)^2|Z_n-Z|^2e_j\\
\le \int_0^1 ds    \int_{\Omega}dx\int_{|v|\ge \frac{4}{\varepsilon}}dv\ (1+|v|)^{-1/2}(|Z_n|^2+|Z|^2)(1+|v|)^{2+1/2}\sqrt{\mu}.
\end{multline*}
Then for a sufficiently small $\varepsilon\ll 1$, we have
\begin{multline*}(1+|v|)^{2+1/2}\sqrt{\mu}\approx(1+|v|)^{2+1/2}\exp(-|v|^2/2)\lesssim \exp(-c|v|^2)\lesssim \exp\left(-\frac{16c}{\varepsilon^2}\right)\lesssim \varepsilon,\end{multline*} for some uniform constant $0<c<\frac{1}{2}.$
Therefore, we have
\begin{multline*}\int_0^1 ds    \int_{S^c_{\varepsilon,0}}dxdv\ (1-\chi_1)^2|Z_n-Z|^2e_j\\
	\le \int_0^1 ds    \int_{\Omega}dx\int_{|v|\ge \frac{4}{\varepsilon}}dv\ (1+|v|)^{-1/2}(|Z_n|^2+|Z|^2)(1+|v|)^{2+1/2}\sqrt{\mu}\\
	\lesssim \varepsilon \sup_{0\le s\le 1}(\|(1+|v|)^{-1/2}Z_n(s)\|^2_2+\|(1+|v|)^{-1/2}Z(s)\|^2_2).
\end{multline*} Note that we have the uniform boundedness
$$\sup_{0\le s\le 1,\ n\ge 1}\|(1+|v|)^{-1/2}Z_n(s)\|^2_{L^2}<C,$$ by \eqref{fnbounded} and that 
$\|(1+|v|)^{-1/2}Z(s)\|^2_2=\|(1+|v|)^{-1/2}Z(0)\|^2_2,$ by the transport equation \eqref{limit eq}.

On the other hand, for the other remainder term in \eqref{near the boundary ineq0}, we observe that
\begin{multline}\label{near the boundary ineq}
   \int_0^1 ds    \int_{S^c_{\varepsilon,1}}dxdv (1-\chi_1)^2|Z_n-Z|^2e_j\\
   \le \int_0^1 ds    \int_{S^c_{\varepsilon,1}}dxdv (1-\chi_1)^2\left(|(I-P)(Z_n-Z)|^2+|PZ_n-PZ|^2\right)e_j\\
   =\int_{S^c_{\varepsilon,1}}dxdv (1-\chi_1)^2\left(|(I-P)Z_n|^2+|PZ_n-PZ|^2\right)e_j,\end{multline} as $(I-P)Z=0.$ Note that by the additional exponential decay $e_j$ with respect to $|v|$, we have
$$\int_{S^c_{\varepsilon,1}}dxdv (1-\chi_1)^2|(I-P)Z_n|^2e_j\lesssim \|(I-P)Z_n\|_\sigma \lesssim \frac{1}{n}.$$In addition, we define
$$PZ_n=a_n(t,x)+\vec{b}_n(t,x)\cdot (e_2,e_3,e_4)+c_n(t,x)e_5,$$ and
$$PZ=a(t,x)+\vec{b}(t,x)\cdot (e_2,e_3,e_4)+c(t,x)e_5,$$ for $\{e_j\}_{j=1,...,5} $ is the orthonormal basis of $\text{span}\{\sqrt{\mu},v\sqrt{\mu},|v|^2\sqrt{\mu}\}.$ 
Note that $a_n$, $b_n$, $c_n$, $a$, $b$, and $c$ are functions of $t$ and $x$. Then, we observe that the remainder term satisfies
\begin{multline}
   \int_{S^c_{\varepsilon,1}}dxdv (1-\chi_1)^2|PZ_n-PZ|^2e_j\\\lesssim  \int_0^1 ds   \int_{\Omega\setminus \Omega_\varepsilon }dx \left(|a_n-a|^2+|\vec{b}_n-\vec{b}|^2 +|c_n-c|^2\right) \int_{ |v\cdot n_x| < \frac{\varepsilon}{2}\text{ or }|v|> \frac{1}{\varepsilon}}dv\ (1+|v|)^{l} \sqrt{\mu}\\
   \lesssim  \int_{ |v\cdot n_x| < \frac{\varepsilon}{2}\text{ or }|v|> \frac{1}{\varepsilon}}dv\ (1+|v|)^{l} \sqrt{\mu},
\end{multline}for some $l\ge 2$ by 
$$\int_0^1\|PZ_n\|^2_\sigma ds \approx \int_0^1 \left(\|a_n(s,\cdot)\|_2^2+\|\vec{b}_n(s,\cdot)\|_2^2+\|c_n(s,\cdot)\|_2^2\right)ds\lesssim 1,
$$and
$$\int_0^1\|PZ\|^2_\sigma ds \approx \int_0^1 \left(\|a(s,\cdot)\|_2^2+\|\vec{b}(s,\cdot)\|_2^2+\|c(s,\cdot)\|_2^2\right)ds\lesssim 1,
$$from the linear independency of $e_j$.
Then, if $|v|>\frac{1}{\varepsilon}$, then  $(1+|v|)^{l}\sqrt{\mu} \le C\varepsilon,\text{ for }|v|>\frac{1}{\varepsilon},$ if $\varepsilon$ is sufficiently smnall. On the other hand, if $|v\cdot n_x|<\frac{\varepsilon}{2},$ we have
$$\int_{|v\cdot n_x|<\frac{\varepsilon}{2}}dv \ (1+|v|)^l \sqrt{\mu} \lesssim \int_{-\frac{\varepsilon}{2}}^{\frac{\varepsilon}{2}}dv_{||}\int_{\mathbb{R}^2}dv_\perp e^{-|v_\perp|^2/8} \lesssim \varepsilon,$$ where $v_{||}\eqdef (n_x\cdot v)n_x,$ and $v_\perp = v-v_{||}\text { for }|n_x\cdot v|\le \frac{\varepsilon}{2}.$ 
Then the (LHS) of \eqref{near the boundary ineq} is bounded from above by
$$\int_0^1 ds    \int_{S^c_{\varepsilon,1}}dxdv (1-\chi_1)^2|Z_n-Z|^2e_j\lesssim \varepsilon.$$ 

\subsection{On the non-grazing set $D^\varepsilon_{ng}$} Finally, we are now left with the $L^2$ norm for the non-grazing set $D^\varepsilon_{ng}$ from \eqref{final2}
$$\int_0^1 ds    \int_{S^c_{\varepsilon,2}}dxdv\ (1-\chi_1)^2|Z_n-Z|^2e_j.$$In this subsection, we will prove that there is no concentration at the boundary, so that we can conclude that $Z_n$ converges strongly to $Z$ in $[0,1]\times \bar{\Omega}\times\rth.$ 
The main strategy in this section is to show that the non-grazing set part $\chi_\pm Z_n$ can be controlled by the inner boundary part $Z_n|_{\gamma_\varepsilon}$, which will be further controlled by the interior compactness. Here the inner boundary is defined as $\gamma^\varepsilon\eqdef \{x:\zeta(x)=-\varepsilon^4\}\times \rth.$
Now we fix $(s,x,v)\in D^\varepsilon_{ng}.$ Then we define backward/forward in time characteristic trajectories $\chi_\pm$ as
\begin{equation}\label{chipm}
    \begin{split}
        \chi_+(t,x,v)&=1_{\Omega\setminus \Omega_\varepsilon}(x-v(t-s))1_{\{|v|\le 1/\varepsilon,\ n_{x-v(t-s)}\cdot v >\varepsilon\}}(v) ,\text{ for }0\le t\le s,\\
        \chi_-(t,x,v)&=1_{\Omega\setminus \Omega_\varepsilon}(x-v(t-s))1_{\{|v|\le 1/\varepsilon,\ n_{x-v(t-s)}\cdot v <-\varepsilon\}}(v) ,\text{ for }0\le s\le t,
    \end{split}
\end{equation}where $\Omega_\varepsilon\eqdef \{x\in\Omega : \zeta(x)\le -\varepsilon^4\}.$ 
Note that $\chi_\pm$ solves the transport equation \newline $(\partial_t+~v\cdot~ \nabla_x)\chi_\pm =0$ with $$\chi_\pm(s,x,v)= 1_{\Omega\setminus \Omega_\varepsilon}(x)1_{\left\{|v|\le 1/\varepsilon,\ n_{x}\cdot v\lessgtr \pm\varepsilon\right\}}(v),$$ and it satisfies the following lemma:
\begin{lemma}[Lemma 10 of \cite{MR2679358}]\label{guolemma10}$\chi_\pm$ satisfies the followings:
 \begin{enumerate}
    \item For $0\le s-\varepsilon^2\le t\le s,$ if $\chi_+(t,x,v)\ne 0$ then $n_x\cdot v>\frac{\varepsilon}{2}>0$. Moreover, $\chi_+(s-\varepsilon^2,x,v)=0,$ for $\zeta(x)\ge-\varepsilon^4.$
    \item For $s\le  t \le s+\varepsilon^2\le 1,$ if $\chi_-(t,x,v)\ne 0,$ then $n_x\cdot v<-\frac{\varepsilon}{2}<0$. Moreover, $\chi_-(s+\varepsilon^2,x,v)=0$, for $\zeta(x)\ge-\varepsilon^4.$ 
\end{enumerate}
\end{lemma}
 We now observe that $\chi_\pm Z_n$ satisfies the following equation
$$(\partial_t +v\cdot \nabla_x )(\chi_\pm Z_n)=-\chi_\pm L[Z_n]
+ \chi_\pm \Gamma(g_n,Z_n).$$ 
We claim that $$\int_{S^c_{\varepsilon,2}} |Z_n|^2dxdv\lesssim \varepsilon,$$ if $n$ is sufficiently large. To see this, we first observe the $L^2$ estimate for $\chi_+$ part over $[s-\varepsilon^2,s]\times S^c_{\varepsilon,2}$ that for the inner boundary $\gamma^\varepsilon\eqdef \{x:\zeta(x)=-\varepsilon^4\}\times \rth$,
\begin{multline*}
    \|\chi_+Z_n(s)\|_{L^2(S^c_{\varepsilon,2})}^2+\int_{s-\varepsilon^2}^s\|\chi_+Z_n(t)\|^2_{\gamma_+}dt-\int_{s-\varepsilon^2}^s\|\chi_+Z_n(t)\|^2_{\gamma^\varepsilon_+}dt\\
    = \|\chi_+Z_n(s-\varepsilon^2)\|_{L^2(S^c_{\varepsilon,2})}^2+\int_{s-\varepsilon^2}^s\|\chi_+Z_n(t)\|^2_{\gamma_-}dt-\int_{s-\varepsilon^2}^s\|\chi_+Z_n(t)\|^2_{\gamma^\varepsilon_-}dt\\
    -2\int_{s-\varepsilon^2}^s(\chi_+ L[Z_n],\chi_+ Z_n)dt+2\int_{s-\varepsilon^2}^s(\chi_+ \Gamma(g_n,Z_n),\chi_+ Z_n)dt,
\end{multline*}where $(\cdot,\cdot)$ is the $L^2$ inner product on $S^c_{\varepsilon,2}$.
By Lemma \ref{guolemma10}, $\chi_+Z_n(s-\varepsilon^2)=0$. Also, $\chi_+Z_n=0$ on $\gamma_-$ and $\gamma_-^\varepsilon$ by the support condition of $\chi_+$. On the other hand, by \eqref{I-P limit} and Lemma 6 of \cite{guo2002landau}, we have
\begin{multline*}
    \int_{s-\varepsilon^2}^s(\chi_+ L[Z_n],\chi_+ Z_n)dt=\int_{s-\varepsilon^2}^s( L[Z_n],\chi^2_+ Z_n)dt\\
    \int_{s-\varepsilon^2}^s(L[(1-\chi^2_++\chi^2_+)Z_n],\chi^2_+ Z_n)dt=\int_{s-\varepsilon^2}^s( L[\chi^2_+Z_n],\chi^2_+ Z_n)dt,\end{multline*}
by the support condition of $\chi_+$. Thus,
\begin{multline*}
\int_{s-\varepsilon^2}^s(\chi_+ L[Z_n],\chi_+ Z_n)dt=\int_{s-\varepsilon^2}^s( L[\chi^2_+Z_n],\chi^2_+ Z_n)dt\\
\le C\int_{0}^1\|(I-P)\chi^2_+Z_n\|_\sigma^2dt\le C\int_{0}^1\|(I-P)Z_n\|_\sigma^2dt=\frac{C}{n}.
\end{multline*}Finally, we observe that, by Theorem 2.8 at (2.16) of \cite{2016arXiv161005346K}, \eqref{gsmall}, and \eqref{fnbounded}, we have
\begin{multline*}
    \int_{s-\varepsilon^2}^s(\chi_+ \Gamma(g_n,Z_n),\chi_+ Z_n)dt=\int_{s-\varepsilon^2}^s( \Gamma(g_n,Z_n),\chi_+^2 Z_n)dt
    \\\le C\|g_n\|_\infty \int_{s-\varepsilon^2}^s\|Z_n\|_\sigma\|\chi_+^2Z_n\|_\sigma dt\le C\|g_n\|_\infty \int_{s-\varepsilon^2}^s\|Z_n\|^2_\sigma dt\le \frac{C}{n}.
\end{multline*}
Altogether, we have
\begin{equation}
      \|\chi_+Z_n(s)\|_{L^2}^2+\int_{s-\varepsilon^2}^s\|\chi_+Z_n(t)\|^2_{\gamma_+}dt-\int_{s-\varepsilon^2}^s\|\chi_+Z_n(t)\|^2_{\gamma^\varepsilon_+}dt\le \frac{C}{n}.
\end{equation} Here, we note that by definition $$\chi_+Z_n(s,x,v)= 1_{\Omega\setminus \Omega_\varepsilon}(x)1_{\{|v|\le 1/\varepsilon,\ n_{x}\cdot v >\varepsilon\}}(v)Z_n(s,x,v).$$
  Similarly, we obtain for the part $\chi_-Z_n$
  \begin{equation}
      \|\chi_-Z_n(s)\|_{L^2}^2+\int^{s+\varepsilon^2}_s\|\chi_-Z_n(t)\|^2_{\gamma_-}dt-\int^{s+\varepsilon^2}_s\|\chi_-Z_n(t)\|^2_{\gamma^\varepsilon_-}dt\le \frac{C}{n}.
\end{equation}
Altogether, we have
\begin{equation}\label{19}
      \|Z_n(s)\|_{L^2(S^c_{\varepsilon,2})}^2
     \le\int_{s-\varepsilon^2}^s\|\chi_+Z_n(t)\|^2_{\gamma^\varepsilon_+}dt+\int^{s+\varepsilon^2}_s\|\chi_-Z_n(t)\|^2_{\gamma^\varepsilon_-}dt + \frac{C}{n}.
\end{equation}

Now we will prove that the right-hand side of \eqref{19} can be arbitrarily small by showing that the right-hand side can further be bounded via the interior compactness inside $S_\varepsilon$. 
In order to control the trace norm on the non-grazing set, we are going to derive a trace theorem for the Landau equation to $1_{\{|v|\le \frac{1}{\varepsilon}\}}(Z_n-Z)$ over the domain $\bar{S}_\varepsilon.$  
We first consider the estimate for $t\in (s-\varepsilon^2,s)$. Recall that $\chi_+$ from \eqref{chipm} indeed satisfies 
\begin{equation}\label{chi+prop}
	\begin{split}
	\partial_t \chi_+ +v\cdot \nabla_x \chi_+&=0,\\
	\chi_+(s-\varepsilon^2,x,v)&=0 \text{ for }\mathrm{dist}(x,\partial\Omega_\varepsilon)\le \varepsilon,
	\end{split}
\end{equation}where $\Omega_\varepsilon\eqdef \{x\in\Omega : \zeta(x)=-\varepsilon^4\}.$ 
We choose a smooth cutoff function $\chi_b=\chi_b^\varepsilon(x)$ near $\partial\Omega_\varepsilon$ such that $\chi_b\equiv 1\text{ if }\mathrm{dist}(x,\partial\Omega_\varepsilon)\le \frac{\varepsilon^4}{4}$,
$\chi_b\equiv 0\text{ if }\mathrm{dist}(x,\partial\Omega_\varepsilon)\ge \varepsilon^4,$ and the growth is up to $|\nabla_x\chi_b|\lesssim \varepsilon^{-3/2}.$ We also choose a smooth cutoff function $\chi_2=\chi_2(v)$ such that $\chi_2 = 1 $ for $|v|\le \frac{1}{\varepsilon}$ and $=0$ for $|v|\ge \frac{4}{\varepsilon}$ and \begin{equation}\label{chi2cutoff}|\chi_+\chi_2|+|\nabla_v(\chi_+ \chi_2)|+|\nabla^2_v (\chi_+\chi_2)|\lesssim \mu\left(\frac{|v|}{4}\right).\end{equation}
Note that $\chi_2(v)$ has a larger support than $1_{|v|\le \frac{1}{\varepsilon}}$. 
We then take $\bar{\chi}=\chi_2\chi_b\chi_+,$ such that $\bar{\chi}(s-\varepsilon^2,x,v)=0$  for $\mathrm{dist}(x,\partial\Omega_\varepsilon)\le \varepsilon$ and
$$(\partial_t+v\cdot \nabla_x)\bar{\chi}=\chi_+\chi_2v\cdot\nabla_x \chi_b.$$ 
Now consider the following rearranged equation \eqref{rearrange Landau} for this argument:$$
\partial_t Z_n+v\cdot\nabla_{x}Z_n=\nabla_{v}\cdot(\sigma_{G_n}\nabla_{v}Z_n )+a_{g_n}\cdot\nabla_{v}Z_n+\bar{K}_{g_n}Z_n,$$where $G_n=\mu+\sqrt{\mu}g_n$.
Then, note that $\bar{\chi}Z_n$ satisfies the equation
  \begin{multline}(\partial_t+v\cdot \nabla_x )(\bar{\chi}Z_n)=\chi_2\chi_+Z_nv\cdot\nabla_x\chi_b+\nabla_{v}\cdot(\sigma_{G_n}\nabla_{v}(\bar{\chi}Z_n))\\
-\sigma_{G_n}Z_n\Delta_v\bar{\chi}-2\sigma_{G_n}\nabla_vZ_n\cdot \nabla_v\bar{\chi}-Z_n\nabla_v(\sigma_{G_n})\cdot \nabla_v\bar{\chi}\\ +\bar{\chi}a_{g_n}\cdot\nabla_{v}Z_n
+\bar{K}_{g_n} (\bar{\chi}Z_n) . \end{multline}
  We multiply $\bar{\chi}Z_n$ and integrate on $(s-\varepsilon^2,s)\times S_\varepsilon$ to obtain that
  \begin{multline}
  \frac{1}{2}\left(\|\bar{\chi}Z_n(s)\|^2_{L^2(S_\varepsilon)}-\|\bar{\chi}Z_n(s-\varepsilon^2)\|^2_{L^2(S_\varepsilon)}\right)+\int_{s-\varepsilon^2}^{s}dt\|\bar{\chi}Z_n\|^2_{\gamma^\varepsilon}\\
= -\int_{s-\varepsilon^2}^{s}dt \iint_{S_\varepsilon} dxdv\ \sigma_{G_n}|\nabla_{v}(\bar{\chi}Z_n)|^2\\+\int_{s-\varepsilon^2}^{s}dt \iint_{S_\varepsilon} dxdv\ \bigg[ \bar{\chi}Z_n\bigg(\chi_2\chi_+Z_nv\cdot\nabla_x\chi_b
-\sigma_{G_n}Z_n\Delta_v\bar{\chi}-2\sigma_{G_n}\nabla_vZ_n\cdot \nabla_v\bar{\chi}\\-Z_n\nabla_v(\sigma_{G_n})\cdot \nabla_v\bar{\chi} +\bar{\chi}a_{g_n}\cdot\nabla_{v}Z_n
+\bar{K}_{g_n} (\bar{\chi}Z_n)\bigg)\bigg],
 \end{multline}
by the integration by parts. Note that $\bar{\chi}Z_n=0$ on $\gamma_-^\varepsilon$ by the support condition of $\chi_+$. By \eqref{chi+prop} and the support condition of $\chi_b$, we also have $\bar{\chi}Z_n(s-\varepsilon^2)=0.$ Thus, we have \begin{multline}
\frac{1}{2}\|\bar{\chi}Z_n(s)\|^2_{L^2(S_\varepsilon)}+\int_{s-\varepsilon^2}^{s}dt\|\bar{\chi}Z_n\|^2_{\gamma^\varepsilon_+}+\int_{s-\varepsilon^2}^{s}dt \iint_{S_\varepsilon} dxdv\ \sigma_{G_n}|\nabla_{v}(\bar{\chi}Z_n)|^2\\
=\int_{s-\varepsilon^2}^{s}dt \iint_{S_\varepsilon} dxdv\ \bigg[ \bar{\chi}Z_n\bigg(\chi_2\chi_+Z_nv\cdot\nabla_x\chi_b
-\sigma_{G_n}Z_n\Delta_v\bar{\chi}-2\sigma_{G_n}\nabla_vZ_n\cdot \nabla_v\bar{\chi}\\-Z_n\nabla_v(\sigma_{G_n})\cdot \nabla_v\bar{\chi} +\bar{\chi}a_{g_n}\cdot\nabla_{v}Z_n
+\bar{K}_{g_n} (\bar{\chi}Z_n)\bigg)\bigg]
, \end{multline}
We estimate the upper bound of each term of the right-hand side.
We first observe that 
\begin{multline*}\left|\int_{s-\varepsilon^2}^{s}dt \iint_{S_\varepsilon} dxdv\  \bar{\chi}Z_n\chi_2\chi_+Z_nv\cdot\nabla_x\chi_b\right|\\\lesssim\int_{s-\varepsilon^2}^{s}dt \iint_{S_\varepsilon} dxdv\ \mu\left(\frac{|v|}{4}\right)|Z_n|^2|\nabla_x\chi_b|\\ \lesssim\varepsilon^{-3/2}\int_{s-\varepsilon^2}^{s}dt\ \|(1+|v|)^{-1/2}Z_n\|^2_{L^2(S_\varepsilon)},\end{multline*}
by the assumption of $\chi_b$. 
Also, by\eqref{Eq : bar K}, Lemma 3 and Lemma 6 of \cite{guo2002landau}, we have
\begin{multline}\left|\int_{s-\varepsilon^2}^{s}dt \iint_{S_\varepsilon} dxdv\  \bar{\chi}Z_n\bar{K}_{g_n} (\bar{\chi}Z_n)\right|\\\le \int_{s-\varepsilon^2}^{s}dt\  (\eta \|\bar{\chi} Z_n\|_{\sigma}+C_\eta \|\bar{\chi}Z_n\|_{L^2(S_\varepsilon)})\|(\bar{\chi}Z_n)\|_\sigma
\\
\le  \int_{s-\varepsilon^2}^{s}dt\  (\eta' \| Z_n\|^2_{\sigma}+C_{\eta'} \|(1+|v|)^{-1/2}Z_n\|^2_{L^2(S_\varepsilon)}),\end{multline} for a sufficiently small $\eta'$
by Young's inequality. 
 We also note that by Lemma 3 of \cite{guo2002landau}, we have
$\sigma^{ij}v_iv_j =\lambda_1 |v|^2,$ where $\lambda_1 \approx (1+|v|)^{-3}.$ Therefore,$$\iint_{S_\varepsilon}dxdv\ |\sigma^{ij}v_iv_j(\bar{\chi}Z_n)^2|\le \iint_{S_\varepsilon}dxdv\ \frac{(\bar{\chi}Z_n)^2}{1+|v|}.$$ 
Here, note that by Lemma 3 of \cite{guo2002landau} and Lemma 2.4 of \cite{2016arXiv161005346K}, if $n$ is sufficiently large so that $\|g_n\|_{L^\infty}\ll 1$, then $$\sigma^{ij}\partial_i(\bar{\chi}Z_n)\partial_j(\bar{\chi}Z_n)\approx \sigma^{ij}_{G_n}\partial_i(\bar{\chi}Z_n)\partial_j(\bar{\chi}Z_n),$$ where $G_n=\mu+\sqrt{\mu}g_n$. 
Then, by \eqref{chi2cutoff}, Lemma 2.4 of \cite{2016arXiv161005346K} and Lemma 3 of \cite{guo2002landau}, we have
\begin{multline*} \left|\int_{s-\varepsilon^2}^{s}dt \iint_{S_\varepsilon} dxdv\  (\bar{\chi}Z_n)\sigma_{G_n}Z_n\Delta_v\bar{\chi}\right|\\ \lesssim \int_{s-\varepsilon^2}^{s}dt \iint_{\zeta<-\varepsilon^4\text{ and }|v|\le \frac{4}{\varepsilon}} dxdv\ \mu\left(\frac{|v|}{4}\right) \frac{|Z_n|^2}{1+|v|}
\\\lesssim \int_{s-\varepsilon^2}^{s}dt \ \|(1+|v|)^{-1/2}Z_n\|^2_{L^2(S_\varepsilon)}.
\end{multline*}
Similarly, we have
\begin{multline*} \left|\int_{s-\varepsilon^2}^{s}dt \iint_{S_\varepsilon} dxdv\  2(\bar{\chi}Z_n)\sigma_{G_n}\nabla_v Z_n\cdot\nabla_v\bar{\chi}\right|\\ \lesssim \int_{s-\varepsilon^2}^{s}dt \iint_{\zeta<-\varepsilon^4\text{ and }|v|\le \frac{4}{\varepsilon}} dxdv\ \mu\left(\frac{|v|}{4}\right)\sigma_{G_n} |Z_n ||\nabla_vZ_n|\\
 \lesssim\eta \int_{s-\varepsilon^2}^{s}dt \iint_{S_\varepsilon} dxdv\ \sigma_{G_n}  |\nabla_vZ_n|^2+C_\eta\int_{s-\varepsilon^2}^{s}dt \iint_{S_\varepsilon} dxdv\ \mu\left(\frac{|v|}{2}\right)\sigma_{G_n} |Z_n |^2\\ \lesssim\eta \int_{s-\varepsilon^2}^{s}dt \iint_{S_\varepsilon} dxdv\ \sigma_{G_n}  |\nabla_vZ_n|^2+C_\eta\int_{s-\varepsilon^2}^{s}dt \iint_{S_\varepsilon} dxdv\ \mu\left(\frac{|v|}{2}\right)\frac{|Z_n |^2}{1+|v|}\\
 \lesssim\eta \int_{s-\varepsilon^2}^{s}dt \iint_{S_\varepsilon} dxdv\ \sigma_{G_n}  |\nabla_vZ_n|^2+C_\eta\int_{s-\varepsilon^2}^{s}dt \ \|(1+|v|)^{-1/2}Z_n\|^2_{L^2(S_\varepsilon)},
\end{multline*} for any small $\eta>0$ by Young's inequality.
In addition, we have
\begin{multline*} \left|\int_{s-\varepsilon^2}^{s}dt \iint_{S_\varepsilon} dxdv\ (\bar{\chi}Z_n) Z_n\nabla_v\sigma_{G_n}\cdot \nabla_v\bar{\chi}\right|\\
\lesssim \int_{s-\varepsilon^2}^{s}dt \iint_{\zeta<-\varepsilon^4\text{ and }|v|\le \frac{4}{\varepsilon}} dxdv\ \mu\left(\frac{|v|}{4}\right) \frac{|(\bar{\chi}Z_n) ||Z_n|}{(1+|v|)^2}\\\lesssim\int_{s-\varepsilon^2}^{s}dt \ \|(1+|v|)^{-1/2}Z_n\|^2_{L^2(S_\varepsilon)}.
\end{multline*} Also, by \eqref{chi2cutoff} and the definition of $a_{g_n}$ from \eqref{Eq : bar A}, we observe that
\begin{multline*} \left|\int_{s-\varepsilon^2}^{s}dt \iint_{S_\varepsilon} dxdv\ (\bar{\chi}Z_n)\bar{\chi} a_{g_n}\cdot\nabla_{v}Z_n\right|\\
\lesssim\int_{s-\varepsilon^2}^{s}dt \iint_{S_\varepsilon} dxdv\ |Z_n||\nabla_{v}Z_n|\left(|\phi^{ij}\ast (v_i\mu^{1/2}g_n)| +|\phi^{ij}\ast (\mu^{1/2}\partial_j g_n)|\right)\\
\lesssim\int_{s-\varepsilon^2}^{s}dt \iint_{S_\varepsilon} dxdv\ |Z_n||\nabla_{v}Z_n|\left(2|\phi^{ij}\ast (\mu^{1/4}g_n)| +|\partial_j\phi^{ij}\ast (\mu^{1/2} g_n)|\right)\\
\lesssim \|g_n\|_{L^\infty}\int_{s-\varepsilon^2}^{s}dt \iint_{\zeta<-\varepsilon^4\text{ and }|v|\le \frac{4}{\varepsilon}} dxdv\ \mu\left(\frac{|v|}{4}\right) \frac{|Z_n||\nabla_{v}Z_n|}{(1+|v|)}\\\lesssim \eta\int_{s-\varepsilon^2}^{s}dt\ \|Z_n\|^2_{\sigma}+\frac{C_\eta}{n^2}\int_{s-\varepsilon^2}^{s}dt\ \|(1+|v|)^{-1/2}Z_n\|^2_{L^2(S_\varepsilon)},\end{multline*}for a sufficiently small $\eta>0$ by Young's inequality.
Altogether, we have\begin{multline*}
     \int_{s-\varepsilon^2}^s\|\chi_+Z_n(t)\|^2_{\gamma^\varepsilon_+}dt
  \\ \lesssim (C_\eta+\varepsilon^{-3/2}) \int_{s-\varepsilon^2}^{s}\left\|(1+|v|)^{-1/2}Z_n\right\|^2_{L^2(S_{\varepsilon})}dt +\eta \int_{s-\varepsilon^2}^{s}dt\ \|Z_n\|^2_\sigma\\ \lesssim (C_\eta+\varepsilon^{-3/2}) \int_{s-\varepsilon^2}^{s}\left\|(1+|v|)^{-1/2}Z\right\|^2_{L^2(S_{\varepsilon})}dt\\+(C_\eta+\varepsilon^{-3/2}) \int_{s-\varepsilon^2}^{s}\left\|(1+|v|)^{-1/2}(Z_n-Z)\right\|^2_{L^2(S_{\varepsilon})}dt+\eta \int_{s-\varepsilon^2}^{s}dt\ \|Z_n\|^2_\sigma,
\end{multline*} for any small $\eta>0$. We repeat the same argument for the part $\int^{s+\varepsilon^2}_s\|\chi_-Z_n(t)\|^2_{\gamma^\varepsilon_-}dt$ of \eqref{19} using $\chi_-$, instead of $\chi_+$. Note that, by the interior compactness, we have for a fixed $\varepsilon>0$
$$\lim_{n\rightarrow \infty}\int_{s-\varepsilon^2}^{s}\left\|(1+|v|)^{-1/2}(Z_n-Z)\right\|^2_{L^2(S_{\varepsilon})}dt=0.$$
 Then, by \eqref{19},
 we have for a small $\eta \sim\sqrt{\varepsilon}$ such that $C_\eta\lesssim \varepsilon^{-3/2}$ and for a sufficiently large $n>0$, \begin{multline*}
      \|Z_n(s)\|_{L^2(S^c_{\varepsilon,2})}^2\lesssim (C_\eta+\varepsilon^{-3/2}) \int_{s-\varepsilon^2}^{s+\varepsilon^2}\left\|(1+|v|)^{-1/2}Z\right\|^2_{L^2(S_{\varepsilon})}dt\\+(C_\eta+\varepsilon^{-3/2}) \int_{s-\varepsilon^2}^{s+\varepsilon^2}\left\|(1+|v|)^{-1/2}(Z_n-Z)\right\|^2_{L^2(S_{\varepsilon})}dt +\eta \int_{s-\varepsilon^2}^{s+\varepsilon^2}dt\ \|Z_n\|^2_\sigma+ \frac{C}{n}\\\lesssim 2(C_\eta+\varepsilon^{-3/2})\varepsilon^2  \sup_{t\in [0,1]}\left\|(1+|v|)^{-1/2}Z(t)\right\|^2_{L^2(S_{\varepsilon})}\\+(C_\eta+\varepsilon^{-3/2})\varepsilon^2+\eta \int_{s-\varepsilon^2}^{s+\varepsilon^2}dt\ \|Z_n\|^2_\sigma+ \frac{C}{n}
      \lesssim C'\sqrt{\varepsilon},
\end{multline*} by \eqref{fnbounded} where $C'>0$ depends on $a_0,c_0,c_1,c_2,b_0,b_1,$ and $\bar{w}$ of Lemma \ref{guolemma6}. Therefore, for any small $\varepsilon>0$, we have\begin{equation}\label{final3} \|Z_n(s)\|_{L^2(S^c_{\varepsilon,2})}^2\lesssim C'\sqrt{\varepsilon},\end{equation} for large $n$.
\subsection{Strong convergence and the non-zero $PZ$}By \eqref{final1}, \eqref{final2}, \eqref{near the boundary ineq0}, \eqref{fnbounded}, \eqref{near the boundary ineq}, and \eqref{final3}, we obtain 
$$\int_0^1 ds \int_\Omega dx\ |\langle Z_n,e_j\rangle - \langle Z,e_j\rangle |^2 \rightarrow 0,$$ where $e_j$ are an orthonormal basis for 
$\text{span}\{\sqrt{\mu},v\sqrt{\mu},|v|^2\sqrt{\mu}\}.$ Since $e_j(v)$ is smooth and the 0$^{th}$ and the 1$^{st}$ derivatives are exponentially decaying for large $|v|$, we obtain that
$$\sum_{1\le j\le 5} \int_0^1 ds \|\langle Z_n,e_j\rangle e_j - \langle Z,e_j\rangle e_j\|_{\sigma}^2  \rightarrow 0.$$ Finally, note that $$Z_n=\sum_{1\le j\le 5}\langle Z_n,e_j\rangle e_j+(I-P)Z_n,$$ and we have \eqref{I-P limit}. Therefore, we obtain the strong convergence of $Z_n$ to $Z$ in $\int_0^1 ds \|\cdot \|^2_{\sigma}$, and we have
$$\int_0^1 ds\  \|PZ \|^2_{\sigma}=1.$$
Also, recall that the specular reflection condition for $Z_n$ is $Z_n(t,x,v)=Z_n(t,x,R_x(v))$. By taking $n\rightarrow \infty$, we can observe that
$Z$ satisfies the same condition for $|v\cdot n_x|\ge \varepsilon/2.$ By continuity of $Z$, we obtain $Z(t,x,v)=Z(t,x,R_x(v))$.
\subsection{ $Z$ is indeed zero.}
On the other hand, we show below that $PZ$ is indeed zero, which will lead us to a contradiction. The proof will be done via the use of the specular boundary conditions, \eqref{limit eq}, and the conservation laws \eqref{conservationlaws} and \eqref{conservationangcon}.
 Recall that, by the conservation laws \eqref{conservationlaws}, we first obtain 
 $$\int Z\sqrt{\mu}=\int Z|v|^2\sqrt{\mu}=0.$$
On the other hand, Lemma \ref{guolemma6} implies that, for any $s\in[0,1]$, we obtain the conservation laws in the form of
\begin{equation}
      \int\left(\left(\frac{c_0}{2}|x|^2-b_0\cdot x +a_0\right)+\left(\frac{c_0t^2}{2}+c_1s+c_2\right)|v|^2\right)\sqrt{\mu}=0,
  \end{equation}
and \begin{equation}
      \int\left(\left(\frac{c_0}{2}|x|^2-b_0\cdot x +a_0\right)|v|^2+\left(\frac{c_0t^2}{2}+c_1s+c_2\right)|v|^4\right)\sqrt{\mu}=0.
  \end{equation}
This implies $c_0=c_1=0$. Also, by the specular reflection condition that $Z(s,x,v)= Z(s,x,R_x(v)),$ we have for any $x\in \partial\Omega$ that $$b\cdot n_x =0 \text{ or } (\bar{w}\times x +b_0s+b_1)\cdot n_x= 0.$$First of all, the coefficient $b_0$ of the time-variable $s$ is zero, which gives
\begin{equation}\label{a4}
    b\cdot n_x =0 \text{ or } (\bar{w}\times x +b_1)\cdot n_x= 0.
\end{equation}
If $\bar{w}=0$, then $b_1\cdot n_x =0$ on $\partial\Omega$. Then we can choose a point $x'\in \partial\Omega$ such that $b_1\ \|\ n_{x'}$ via taking the minimizer of $\min_{\zeta(x)}b_1\cdot x$. Then this gives $b_1\cdot n_{x'}=0$ and $b_1=0.$
If $\bar{w}\ne 0,$ then we decompose $b_1$ as 
$$ b_1=\beta_1 \frac{\bar{w}}{|\bar{w}|}+\beta_2\eta,$$ where $|\eta|=1$ and $\eta\perp \bar{w}.$ Then $$\eta=\left(\frac{\bar{w}}{|\bar{w}|}\times \eta\right)\times \frac{\bar{w}}{|\bar{w}|}.$$ Therefore, we get
$$ b_1=\beta_1 \frac{\bar{w}}{|\bar{w}|}+\beta_2\left(\frac{\bar{w}}{|\bar{w}|}\times \eta\right)\times \frac{\bar{w}}{|\bar{w}|}=\beta_1 \frac{\bar{w}}{|\bar{w}|}-x_0\times \bar{w},$$where $x_0=-\beta_2\left(\frac{\bar{w}}{|\bar{w}|}\times \eta\right)\frac{1}{|\bar{w}|}$. Therefore, by \eqref{a4} we have
$$\beta_1 \frac{\bar{w}}{|\bar{w}|}n_x +((x-x_0)\times \bar{w})\cdot n_x =0.$$
Now note that we can choose a point $x'\in \partial\Omega$ such that $\bar{w}\ \|\ n_{x'}$. Then we deduce $\bar{w}\times (n_{x'}\times (x'-x_0)=0$ and obtain $\beta_1=0.$ Therefore, we obtain 
$$Z=\bar{w}\times (x-x_0)\cdot v\sqrt{\mu}$$ and  $\bar{w}\times(x-x_0)\cdot n_x=0.$ 
If $\Omega$ is not rotationally symmetric, then no nonzero $\bar{w}$ and $x_0$ exist, which provides $Z=0$ from the former case that $\bar{w}=0$. If $\Omega$ is indeed rotationally symmetric and there are nonzero $\bar{w}$ and $x_0$ such that 
$$Z=\bar{w}\times (x-x_0)\cdot v\sqrt{\mu}\text{  and  }\bar{w}\times(x-x_0)\cdot n_x=0.$$ Now we use the conservation of total angular momentum \eqref{conservationangcon} that
$$\int_{\Omega\times \rth}((x-x_0)\times \bar{w})\cdot Zv\sqrt{\mu}dxdv =0,$$ which is equivalent to say
$$\int_{\Omega\times \rth}(\bar{w}\times (x-x_0)\cdot v)^2\mu dxdv =0.$$ 
Therefore, $ \bar{w}\times (x-x_0)\cdot v=0$. Thus we conclude that $Z=0$ and this leads to a contradiction.

\end{proof}
  This finishes the proof for the positivity on a fixed time interval $[0,1]$. In the next section, we prove the main $L^2$ decay theorem in the interval $[0,t]$.

\section{Proof of Theorem \ref{linear l2}}
We are now ready to prove our main theorem on the $L^2$ decay estimates for the solutions $f$ to \eqref{linear}.
\begin{proof}[Proof of Theorem \ref{linear l2}]
Define \begin{equation}\label{T}T=\sup_t \left(t:\sup_{0\le s\le t}\mathcal{E}_\vartheta (f(s))\le 1\right)>0,\end{equation} for some $\vartheta\ge 0.$
For $0\le t\le T,$ let $0\le N\le t\le N+1,$ for some non-negative integer $N$. We split $[0,t]=\left(\cup_{j=0}^{N-1}[j,j+1]\right)\cup [N,t].$
	On each interval $[j,j+1]$ for $j=0,1,...,N-1,$ we define 
	$f^j(s,x,v)\eqdef f(s+j,x,v).$
Then clearly $f^j(s,x,v)$ is a weak solution of \eqref{linear}-\eqref{conservationangcon} on the time interval $s\in[0,1]$ with the new initial condition $f^j(0,x,v)= f(j,x,v)$. Note that since we only consider $t\in [0,T]$ for $T$ from \eqref{T}, $\mathcal{E}_\vartheta(f^j(0))$ is uniformly bounded from above. 
We take the $L^2$ energy estimate over $0\le s\le N$ to obtain 
$$
\|f(N)\|_{2}^2 +\int_0^N ds\ \left(Lf,f\right)
=\|f(0)\|^2_{2} +\int_0^N ds\ (\Gamma(g,f),f),
$$by the specular reflection boundary condition. Equivalently, we have
$$
\|f(N)\|_{2}^2 +\sum_{j=0}^{N-1}\int_0^1 ds\ \left(Lf^j,f^j\right)
=\|f(0)\|^2_{2} +\int_0^N ds\ (\Gamma(g,f),f).
$$Then we use Proposition \ref{main coercivity} and obtain 
$$
\|f(N)\|_{2}^2 +\sum_{j=0}^{N-1}\delta_{\epsilon,j}\int_0^1 ds\ \|f^j\|^2_\sigma 
\le \|f(0)\|^2_{2} +\int_0^N ds\ (\Gamma(g,f),f).
$$ Thus,
\begin{equation}\label{0N ineq}
\|f(N)\|_{2}^2 +\min_{\{j=0,...,N-1\}}\delta_{\epsilon,j}\int_0^N ds\ \|f\|^2_\sigma 
\le \|f(0)\|^2_{2} +\int_0^N ds\ (\Gamma(g,f),f).
\end{equation}
By Theorem 2.8 of \cite{2016arXiv161005346K}, we obtain the energy inequality over $[0,N]$
\begin{equation}\label{0t ineq}
	\|f(N)\|_{2}^2 +\min_{\{j=0,...,N-1\}}\delta_{\epsilon,j}\int_0^N ds\ \|f(s)\|^2_\sigma 
	\le \|f(0)\|^2_{2} +C_0\int_0^N ds\ \|g(s)\|_\infty \|f(s)\|_{\sigma}^2.
\end{equation}
This completes the derivation of the energy inequality for the base case $\vartheta=0$ in the interval $[0,N]$. 
For $\vartheta\ge 0,$ we multiply $(1+|v|)^{2\vartheta}(v) f(t,x,v)$ and take the $L^2$ energy estimate over $0\le s\le N$ to obtain 
$$
	\|f(N)\|_{2,\vartheta}^2 +\int_0^N ds\ \left((1+|v|)^{2\vartheta}Lf,f\right)
=\|f(0)\|^2_{2,\vartheta} +\int_0^N ds\ ((1+|v|)^{2\vartheta}\Gamma(g,f),f),
$$by the specular reflection boundary condition. 
By Lemma 2.7 and Theorem 2.8 of \cite{2016arXiv161005346K}, we have for some $C_\vartheta>0$
\begin{multline}\label{vartheta energy}
\|f(N)\|_{2,\vartheta}^2 +\int_0^Nds\ \left(\frac{1}{2}\|f(s)\|_{\sigma,\vartheta}^2 -C_\vartheta \|f(s)\|^2_\sigma\right)
\\ \le \|f(0)\|^2_{2,\vartheta} +C_\vartheta \int_0^N ds\  \|g(s)\|_\infty \|f(s)\|_{\sigma,\vartheta}^2.
\end{multline}
This completes the derivation of the energy inequality for $\vartheta\ge 0$ in the interval $[0,N].$ Therefore, by the ingredients \eqref{0t ineq} for the base case $\vartheta=0$ and \eqref{vartheta energy} for a general $\vartheta\ge 0,$ we obtain (4.36) of \cite{2016arXiv161005346K} by the same proof via the induction on $\vartheta$ for $\eta \equiv 0$, $s=0$ and $t=N$. Then by the same proof of Theorem 1.2 of \cite{2016arXiv161005346K}, we obtain \eqref{Eq : energy estimate linear} and \eqref{Eq : decay estimate linear} in the time interval $s\in [0,N]$;
for any $\vartheta\in 2^{-1}\mathbb{N} \cup\{0\}$ and  $k\in\mathbb{N}$, there exist $C$ and $\epsilon=\epsilon(\vartheta)>0$ such that 
\begin{equation}\notag
	\sup_{0 \le s\le N}\mathcal{E}_{\vartheta}(f(s)) \le C 2^{2\vartheta} \mathcal{E}_{\vartheta}(f_0),
\end{equation}
and
\begin{equation}	\notag
	\|f(N)\|_{2,\vartheta} \le C_{\vartheta,k} \left(\mathcal{E}_{\vartheta+\frac{k}{2}}(0) \right)^{1/2}\left(  1+ \frac{N}{k}\right)^{-k/2}.%
\end{equation}
Now we consider the local interval $[N,t]$ where we have $0\le t-N\le 1$ and $t\le T$. We recall that if $\|g\|_{L^\infty_m}\le \epsilon$ for a sufficiently small $\epsilon,$ we have
\begin{equation}\label{local ineq}\|(1+|v|)^\vartheta f(t)\|_{L^2}^2+\int_{N}^t \|f(s)\|^2_{\sigma,\vartheta}\  ds\le C e^{t-N} \|(1+|v|)^\vartheta f(N)\|^2_{L^2}, \end{equation}
by \eqref{fnbounded} for $l=\vartheta$ on $[N,t]$. Note that \eqref{local ineq} holds for a solution to \eqref{linear} under \eqref{gsmall ep} and \eqref{ib}-\eqref{conservationangcon} by the local $L^2$ energy inequality and the Gr\"onwall inequality as in \eqref{fnbounded} and we do not need the additional assumption \eqref{contradiction} for \eqref{fnbounded}. Then we observe that \begin{equation}\notag
\mathcal{E}_{\vartheta}(f(t))\le Ce^{t-N}\mathcal{E}_{\vartheta}(f(N))\le C'e^{t-N} 2^{2\vartheta} \mathcal{E}_{\vartheta}(f_0)\le C'e 2^{2\vartheta} \mathcal{E}_{\vartheta}(f_0),
\end{equation}for some $C'>0$ 
and
\begin{multline*}	\notag
\|f(t)\|_{2,\vartheta}\le Ce^{t-N}	\|f(N)\|_{2,\vartheta} \le Ce^{t-N}C_{\vartheta,k} \left(\mathcal{E}_{\vartheta+\frac{k}{2}}(0) \right)^{1/2}\left(  1+ \frac{N}{k}\right)^{-k/2}\\
\le CeC_{\vartheta,k} \left(\mathcal{E}_{\vartheta+\frac{k}{2}}(0) \right)^{1/2}2^{k/2}\left(  1+ \frac{t}{k}\right)^{-k/2},
\end{multline*}
since $$\left(  1+ \frac{N}{k}\right)^{-k/2}\le 2^{k/2}\left(  1+ \frac{t}{k}\right)^{-k/2} ,$$ for $N\le t\le N+1$ and $k\ge 1$.
Therefore, we obtain \eqref{Eq : energy estimate linear} and \eqref{Eq : decay estimate linear} for the time interval $ [0,t]$ for any $0\le t\le T$ where $T$ is defined as \eqref{T}.

We finally choose initially 
$$\mathcal{E}_\vartheta (f_0)\le \epsilon_0 \le \frac{1}{2C2^{2\vartheta}},$$ and we define 
$$T_2 =\sup_t \left(t:\sup_{0\le s\le t}\mathcal{E}_\vartheta (f(s))\le \frac{1}{2}\right)>0.$$
Since $0\le t\le T_2\le T$, we have from \eqref{Eq : energy estimate linear} that
$$	\sup_{0 \le s\le  T}\mathcal{E}_{\vartheta}(f(s)) \le C 2^{2\vartheta} \mathcal{E}_{\vartheta}(f_0)\le \frac{1}{2}.$$
Thus, we deduce that $T_2=\infty$ from the continuity of $\mathcal{E}_\vartheta$, and the theorem follows.
\section*{Acknowledgement}We thank Hongjie Dong for many helpful discussions.

\end{proof}
\bibliographystyle{hplain}

\bibliography{Landau}{}
 	\end{document}